\documentclass[12 pt]{amsart}
\usepackage{hyperref}
\usepackage[colorinlistoftodos]{todonotes}

\usepackage{mathpazo} 
\usepackage{tensor}
\usepackage{setspace}
\usepackage{amsxtra}

\usepackage{geometry}
\usepackage{amssymb}
\usepackage{graphicx}
\usepackage[mathscr]{eucal}
 
\theoremstyle{plain}
\newtheorem{Thm}{Theorem}
\newtheorem{Cor}[Thm]{Corollary}

\newtheorem{Lem}[Thm]{Lemma}
\newtheorem{Prob}[Thm]{Problem}
\newtheorem{Def}[Thm]{Definition}
\newtheorem{Q}[Thm]{Question}
\theoremstyle{definition}
\newtheorem{Remark}[Thm]{Remark}

\newtheorem{Ex}[Thm]{Example}
\newtheorem*{Add1}{Addendum to Theorem \ref{T:FirstVarQuoKappa}}

\renewcommand{\Re}{\operatorname{Re}}
\renewcommand{\Im}{\operatorname{Im}}
\renewcommand{\bar}{\overline}
\renewcommand{\tilde}{\widetilde}
\newcommand{\intl}{\int\limits}
\newcommand{\C}{\mathbb C}
\newcommand{\R}{\mathbb R}
\newcommand{\CP}{\mathbb{CP}}
\newcommand{\Z}{\mathbb Z}
\newcommand{\dee}{\partial}
\newcommand{\deebar}{\overline\partial}
\newcommand{\st}{\,:\,}

\newcommand{\meas}{\mathcal F}
\newcommand{\bmeas}{\mathcal B}
\newcommand{\vol}{\mathcal V}
\DeclareMathOperator{\Vol}{Vol}
\newcommand{\quo}{\mathcal Q}
\newcommand{\opa}{{\mathcal L}_1}
\newcommand{\opb}{{\mathcal L}_2}

\newcommand{\eps}{\varepsilon}
\newcommand{\bndry}{b}
\newcommand{\w}{\wedge}
\numberwithin{equation}{section}

\begin{document}
\title[Variational problems for Fefferman's measure]
{Remarks on variational problems for Fefferman's measure}
\author[David E. Barrett]{David E. Barrett}
\address{Department of Mathematics\\University of Michigan
\\Ann Arbor, MI  48109-1043  USA }

\author[Christopher Hammond]{Christopher Hammond}
\address{Department of Mathematics \\
Texas A\&M University \\
 College Station, TX 77843-3368  USA }


\thanks{The first author was supported in part by NSF grant number DMS-0901205.  The second author was supported in part by RTG grant number DMS-060219.}


\date{\today}

\begin{abstract}  
We investigate the Plateau and isoperimetric problems associated to Fefferman's measure for strongly pseudoconvex real hypersurfaces in $\C^n$ (focusing on the case $n=2$), showing in particular that the isoperimetric problem shares features of both the euclidean isoperimetric problem and the corresponding problem in Blaschke's equiaffine geometry in which the key inequalities are reversed.  

The problems are invariant under constant-Jacobian biholomorphism, but we also introduce  a non-trivial  modified isoperimetric quantity invariant under general biholomorphism.
\end{abstract}

\maketitle

\section{Introduction}\label{S:Intro}

\subsection{Fefferman's measure}

Let $Z\subset \C^n$ be a strongly pseudoconvex hypersurface with defining function $\rho$.  Fefferman's measure of $Z$ is defined by
\begin{equation}\label{E:FeffMeasInt}
\meas(Z) = \intl_Z \sigma_Z,
\end{equation}
where $\sigma_Z$ is the positive $(2n-1)$-form on $Z$ uniquely determined by
\begin{equation}\label{E:FeffMeasForm}
\sigma_Z\w d\rho =  2^{ 2n/(n+1)} \,M(\rho)^{1/(n+1)}
\omega_{\C^n},
\end{equation}
where $\omega_{\C^n}$ is the euclidean volume form and
\begin{equation}\label{E:MADef}
M(\rho)=- \det
\begin{pmatrix}
0 & \rho_{z_j}\\
\rho_{z_{\bar k}} & \rho_{z_{j}\bar z_{  k}}
\end{pmatrix},
\end{equation}
with the subscripts denoting differentiation.  (See \S \ref{SS:Fef-def-note} below regarding the choice of dimensional constant.)

The form $\sigma_Z$ does not depend on the choice of defining function $\rho$, and $\sigma_Z$ obeys the transformation law
\begin{equation*}\label{E:ftl}
H^* \sigma_{H(Z)} =\left| \det H' \right|^{2n/(n+1)} \sigma_Z 
\end{equation*}
when $H$ is a biholomorphic map defined on a neighborhood of  $Z$ (or  a CR diffeomorphism defined on $Z$).  In particular we have 
\begin{equation}\label{E:ConstJacTrans}
\meas(H(Z)) = \left| \det H' \right|^{2n/(n+1)} \meas(Z)
\end{equation}
when $\det H'$ is constant; thus Fefferman measure is preserved by volume-preserving biholomorphic maps.
(See [Fef1, p. 259], [Fef2] and [Bar1] for details and additional information.)

\subsection{The isoperimetric quotient} \label{SS:IsoQuo} 

Assume further that $Z$ is the compact boundary of a bounded domain.  Denote by $\vol(Z)$ the volume enclosed by $Z$ and define the isoperimetric quotient of $Z$ by
\begin{align*}
\quo(Z) = \frac{\meas(Z)^{(n+1)/n}}{\vol(Z)}.
\end{align*}

It follows from \eqref{E:ConstJacTrans} that $\quo(H(Z))=\quo(Z)$ when $H$ is a biholomorphic map with constant Jacobian; in particular this holds when $H$ is a volume-preserving biholomorphic map.

When $Z$ is the unit sphere we have 
$\meas(Z)=\frac{2^{2n/(n+1)}\pi^{n}}{(n-1)!}$, 
$\vol(Z)=\frac{\pi^{n}}{n!}$ 
and so 
$\quo(Z)=\frac{4\pi n}{\sqrt[n]{(n-1)!}}$.  
In view of the transformation law from the previous paragraph we also have $\quo(Z)=\frac{4\pi n}{\sqrt[n]{(n-1)!}}$ when $Z$ is a constant-Jacobian holomorphic image of a sphere.

\subsection{Plan of paper}  \S \ref{S:spec} contains some remarks on the relation of the topics under discussion to planar euclidean geometry and to Blaschke's equiaffine geometry. 

The remaining sections focus entirely on the case $n=2$ with the exception of a brief discussion of higher dimension in \S \ref{S:inv-iso}.

The analogue of Plateau's problem for $\meas$ is discussed in \S \ref{S:maxhyp}, while the isoperimetric problem for $\quo$ is discussed in \S \ref{S:iso}.  In \S \ref{S:inv-iso} the isoperimetric quotient $\quo$ is modified to  obtain a modified quantity $\quo^*$ that is also invariant under biholomorphic mapping with non-constant Jacobian.

Various discussions involving normalization issues are collected in \S \ref{S:normalize}.

\subsection{Note} In some cases we have verified ``routine computations" below by checking that Mathematica evaluates the difference between the left and right sides of the equation to zero.

\medskip

We thank Xiaojun Huang for helpful remarks.

\section{Special cases and relations to other geometric theories} \label{S:spec}

In complex dimension one the assumption of strong pseudoconvexity is vacuous and Fefferman's measure coincides with euclidean arc length.  As is universally known, line segments minimize $\meas(Z)$ among curves joining two fixed endpoints and circles minimize $\quo(Z)$ among simple closed curves; thus $Q(Z)\ge 4\pi$ for all simple closed $Z$.  

In the case of so-called tubular hypersurfaces invariant under purely imaginary translation then Fefferman's measure essentially coincides with 
Blaschke's equiaffinely invariant surface area (as described for example in [Cal, Prop. 1.1]).   Since such hypersurfaces are never bounded, we move the field of action temporarily to the quotient space $A^n=\C^n/i\Z^n$.  Strongly pseudoconvex tubular  hypersurfaces $Z$ in $A^n$ take the form $Z'\times i(\R^n/\Z^n)$ with $Z'$ a strongly convex hypersurface in $\R^n$; moreover, $\meas(Z)$ is Blaschke's measure $\bmeas(Z')$ of $Z'$ (invariant under volume-preserving affine self-maps of $\R^n$).

When $n=2$ we have that $Z'$ is a strongly convex curve and $\bmeas(Z')$ is given by $\intl_{B(Z)} \sqrt[3]{\kappa}\,ds$, where $\kappa$ is euclidean curvature and $ds$ is euclidean arc length.  Then it is known that parabolic arcs maximize $\bmeas(Z')$ within the corresponding isotopy classes of strongly convex curves with location and tangent direction of endpoints fixed [Bla, \S 16], and also that ellipses maximize 
$\frac{\bmeas(Z')^3}{\vol(Z')}$ among simple closed strongly convex curves   [Bla, \S 26].

The isoperimetric quotient $\quo(Z)$ defined in \S \ref{SS:IsoQuo} is not invariant under (quotients of) dilations acting on $A^2$, so there is no universal positive lower bound or finite upper bound for $\quo(Z)$ in this setting.  The modified quotient $\frac{\meas(Z)^3}{\vol(Z)}$ does have the proper invariance properties, however, and the real isoperimetric result quoted above  implies  that $\frac{\meas(Z)^3}{\vol(Z)}\le 8\pi^2$ for all compact tubular strongly pseudconvex $Z$ in $A^2$ with equality holding only for tubes over ellipses.

Similarly we may use the higher-dimensional affine isoperimetric inequality [Hug] to deduce that 
$\frac{\meas(Z)^{\frac{n+1}{n-1}}}{\vol(Z)}\le \frac{n^{\frac{n+1}{n-1}}\pi^{\frac{n}{n-1}}}{\left(\Gamma\left(\frac{n}{2}+1\right)\right)^{\frac{2}{n-1}}}$
for strongly pseudoconvex tubular $Z\subset A^n$.

\section{Maximal hypersurfaces} \label{S:maxhyp}

\subsection{First variation}\label{SS:FirstVarMax}  We wish to identify the hypersurfaces $Z$ that are stationary for  $\meas(Z)$ with respect to compactly-supported perturbations.  This is a local matter, so we focus on the case where $Z$ is given as a graph
\begin{equation}\label{E:GraphSurf}
v=F(z,u)
\end{equation}
over an open base $B\subset \C\times\R$, where $(z,w)=(z,u+iv)$ are standard coordinates on $\C^2$. Setting 
\begin{equation}\label{E:SpecRho}
\rho(z,w)= -v+F(z,u)
\end{equation}
and applying \eqref{E:FeffMeasInt}, \eqref{E:FeffMeasForm} \eqref{E:MADef} we obtain 
\begin{equation}\label{E:FeffGraph}
\meas(Z) = 2^{2/3}
\intl_B \left(\mu(F)\right)^{1/3}\, dV
\end{equation}
where 
\begin{equation*}
\mu(F)=F_{z\bar{z}}(F_{u}^{2}+1)-F_{zu}(F_{u}+i)F_{\bar{z}}-F_{\bar{z}u}(F_{u}-i)F_{z}+F_{uu}|F_{z}|^{2}
\end{equation*}
and $dV$ denotes euclidean volume on $B$.

To compute the first variation we set
$F_\eps=F_0+\eps \mathring{F}$ where $\mathring{F}$ has compact support.  Then the corresponding hypersurfaces $Z_\eps$ satisfy
\begin{align*}
\meas(Z_\eps) =  \meas(Z_0) -
\eps\cdot\frac{2^{2/3}}{3}\intl_B 
\opa(F_0) \mathring F\,dV
+O(\eps^2),
\end{align*}
 where 
\begin{align*}
\opa(F) = & 
\frac{\partial}{\partial u}\left(2 \mu(F)^{-2/3}F_{u}F_{z\bar{z}}\right)
-\frac{\partial^{2}}{\partial z \,\partial \bar{z}}\left(\mu(F)^{-2/3}(F_{u}^{2}+1)\right)\\
&-\frac{\partial}{\partial \bar{z}}\left(\mu(F)^{-2/3} F_{zu}(F_{u}+i)\right)
-\frac{\partial}{\partial u}\left(\mu(F)^{-2/3} F_{zu} F_{\bar{z}}\right) \\
&+\frac{\partial^{2}}{\partial z \,\partial u}\left(\mu(F)^{-2/3} (F_{u}+i) F_{\bar{z}}\right)
-\frac{\partial}{\partial z}\left(\mu(F)^{-2/3}F_{\bar{z}u} (F_{u}-i)\right)
\\
&-\frac{\partial}{\partial u}\left(\mu(F)^{-2/3}F_{\bar{z}u} F_{z}\right)
+\frac{\partial^{2}}{\partial \bar{z} \,\partial u}\left(\mu(F)^{-2/3}(F_{u}-i)F_{z}\right)\\
&+\frac{\partial}{\partial \bar{z}}\left(\mu(F)^{-2/3} F_{uu} F_{z}\right) 
+\frac{\partial}{\partial z}\left(\mu(F)^{-2/3} F_{uu} F_{\bar{z}}\right)
-\frac{\partial^{2}}{\partial u \,\partial u}\left(\mu(F)^{-2/3} F_{z} F_{\bar{z}}\right).
\end{align*}
(Here integration by parts was used to take derivatives off of $\mathring F$.)  

Thus we have the following.
\begin{Thm}[{[Ham, Theorem 16]}]\label{T:FirstVarMaxGraph}
A hypersurface in graph form \eqref{E:GraphSurf} is stationary for $\meas$ (with respect to compactly supported perturbations) if and only if $F$ satisfies $\opa(F)\equiv0$.
\end{Thm}

If the base $B$ is bounded with smooth boundary then the same result holds for perturbations satisfying homogeneous Dirichlet and Neumann boundary conditions.

\subsection{Geometric interpretation}\label{SS:ChrisThesis}
The paper [Ham] contains the following result.
\begin{Thm}\label{T:VPNormalForm}
Let $p\in Z$, where $Z$ is a strongly pseudoconvex hypersurface in $\C^2$.  Then there is a volume-preserving biholomorphic map defined on a neighborhood of $p$ taking $p$ to $0$ and taking $Z$ to a hypersurface of the form
\begin{equation*}\label{E:VPNormalForm}
v=|z|^{2} +\kappa |z|^{4}+\gamma z^{3} \bar{z} + \gamma z \bar{z}^{3} 
+ O\left(|z|^5+|u||z|^3+u^2|z|+|u|^3
\right)
\end{equation*}
with $\kappa\in \R$, $\gamma\in\R_{\ge 0}$.  The quantities $\kappa$ and $\gamma$ are uniquely determined by $p$ and $Z$.
\end{Thm}

Thus $\kappa$ and $\gamma$ define functions on $Z$.  
Up to multiplicative constants these quantities correspond, respectively, to the curvature and (the size of) the torsion invariant in Webster's pseudo-Hermitian geometry [Web].  (See \S \ref{SS:web-curv} concerning the values of the constants.)

When $Z$ is in graph form \eqref{E:GraphSurf} it is furthermore shown in [Ham] that
\begin{equation*}
\kappa = \frac{3}{8} \,\opa(F).
\end{equation*}

\begin{Cor}\label{C:FirstVarMaxKappa} 
A strongly pseudoconvex hypersurface $Z\subset\C^2$  is stationary for $\meas$ (with respect to compactly supported perturbations) if and only if the invariant $\kappa$ described above vanishes identically on $Z$.
\end{Cor}

\subsection{Examples}  For so-called ``rigid" hypersurfaces of the form 
\begin{equation}\label{E:Rigid}
v=F(z)
\end{equation}
we have
\begin{equation*}
\opa(F)=-\big(F_{z\bar z}^{-2/3}\big)_{z\bar z}.
\end{equation*}
Thus a hypersurface of the form \eqref{E:Rigid} will be stationary for $\meas$ if and only if $F_{z\bar z}^{-2/3}$ is harmonic.  Specific examples include the Heisenberg group
\begin{equation}\label{E:Heisenberg}
v=|z|^2
\end{equation}
as well as the hypersurface
\begin{equation}\label{E:HeisVar}
v=\sqrt{z^{-2}+\bar z^{-2}}.
\end{equation}
These examples are locally biholomorphic via the (nonconstant-Jacobian) map \[(z,w)\mapsto (w,4iz^{-2}+iw^2).\]

\subsection{Second variation} 

\subsubsection{Heisenberg group}
Fix an open $B\subset\subset\C\times\R$ and pick (real-valued) $\mathring F\in C^2_0(B)$.  Let $Z_\eps$ be the graph hypersurface over $B$ defined by $v=|z|^2+\eps \mathring F (z,u)$; thus $Z_0$ is an open subset of the Heisenberg group.  

We introduce 
the vector fields 
\begin{align}\label{E:VFHeis}
L &= \frac{\dee}{\dee z} + i\bar z \frac{\dee}{\dee u}\notag\\
\bar L &= \frac{\dee}{\dee \bar z} - i z \frac{\dee}{\dee u}\\
T &= \frac{\dee}{\dee u}\notag
\end{align}
on $\C\times\R$ satisfying 
\begin{align}\label{E:HeisVFRelations}
L\bar L - \bar L L &= -2i T\notag\\
LT &= TL\\
\bar LT &= T\bar L.\notag
\end{align}  
(For motivation, see for example  [JL2].)

By routine computation we verify that
\begin{equation}\label{E:Heis2VarV1}
\meas(Z_\eps) =  \meas(Z_0) -
\frac{\eps^2}{9}
\intl_B \opb(\mathring F)\,dV
+O(\eps^3),
\end{equation}
where
\begin{align*}
\opb(F) =&
L\bar L F\cdot \bar L L F\\
&+3i\left(LTF\cdot\bar L F-\bar L T F\cdot LF\right)\\
& + 2 TF\cdot\left( L\bar L + \bar L L\right)F\\
&+6 TF \cdot F_{z\bar z}.
\end{align*}

We note that $L$, $\bar L$, $T$, $\frac{\dee}{\dee z}$ and $\frac{\dee}{\dee \bar z}$ are all divergence free; thus when integrating by parts we pick up a minus sign but no lower-order terms.  

Integrating by parts three times we find that
$
2\intl_B TF\cdot\left( L\bar L + \bar L L\right)F\,dV
$
and
$
6\intl_B  TF \cdot F_{z\bar z}\,dV
$
are both equal to their own negatives and thus must vanish. Furthermore, integrating by parts once and using \eqref{E:HeisVFRelations} we find that
\begin{align}
3i\intl_B \left(LTF\cdot\bar L F-\bar L T F\cdot LF\right)\,dV
&= -3i\intl_B TF\cdot \left(L\bar L-\bar L L\right)F\,dV\notag\\
&= 6\intl_B TF\cdot TF\,dV.\label{E:iTerm}
\end{align} 
So \eqref{E:Heis2VarV1} can be revised to read
\begin{equation}\label{E:Heis2VarV2}
\meas(Z_\eps) =  \meas(Z_0) -
\frac{\eps^2}{9}
\intl_B \left( |L\bar L \mathring F|^2 + 6|T\mathring F|^2\right)
 \,dV
+O(\eps^3).
\end{equation}

Thus the second variation form is negative semi-definite.
In fact we have the following.
\begin{Thm}\label{T:HeiVarNeg}
The second variation form $-\frac{1}{9} \intl_B \left( |L\bar L F|^2 + 6|TF|^2\right)\,dV$ for the problem described above is negative definite for (real-valued) $F\in C^2_0(B)$.
\end{Thm}

\begin{proof}
It suffices to show that $L\bar L F\equiv 0$ and $TF\equiv 0$ imply $F\equiv 0$.  But these hypotheses imply that
$\intl_B |LF|^2\,dV = -\intl_B F\cdot  L \bar L F\,dV \equiv 0$ and thus $LF\equiv 0$; a similar computation shows that $\bar LF\equiv 0$.  Thus all first derivatives of $F$ vanish implying that $F\equiv 0$ as claimed.
\end{proof}

\subsubsection{Partial global result}  The results of the previous section indicate that (bounded open subsets of) the Heisenberg group \eqref{E:Heisenberg} are 
locally maximal  for Fefferman measure $\meas$ with respect to any smooth compactly-supported finite-dimensional family   of perturbations.   In fact we can say more.

\begin{Thm}\label{T:HeisSemiGlobal}
Suppose that $\tilde F\in C^2_0(B)$ satisfies 
\begin{equation}\label{E:AdHocHyp}
\tilde F_{z\bar z}\le \frac{1}{3}.
\end{equation}
Let 
\begin{align*}
Z_0&=\{(z,u+iv)\st (z,u)\in B, v=|z|^2\}\\
Z&=\{(z,u+iv)\st (z,u)\in B, v=|z|^2+\tilde F(z,u)\}.
\end{align*}
Then $\meas(Z)\le\meas(Z_0)$.
\end{Thm}

\begin{proof}\,[Compare proof of Proposition 3 in [Bol].]
Let $F(z)=|z|^2+\tilde F(z,u)$. Using H\" older's inequality we have
\begin{equation}\label{E:Holder1}
\meas(Z)
=
2^{2/3}
\intl_B \left(\mu(F)\right)^{1/3}\, dV 
\le (2\Vol(B))^{2/3}
\left(\intl_B \mu(F)\, dV\right)^{1/3}.
\end{equation}

By routine computation  we verify that
\begin{align*}
\mu(F)
=&
1 + \tilde F_u^2 + 3 \tilde F_{z\bar z}\tilde F_u^2\\
&-iLT\tilde F\cdot\bar L \tilde F+i\bar L T \tilde F\cdot L\tilde F\\
&+\frac{1}{2}\left(L\bar L + \bar L L\right)\tilde F\\
&+\frac{\dee}{\dee u}
\left(
2\tilde F - \tilde F_z\tilde F_{\bar z}+ \tilde F_z\tilde F_{\bar z}\tilde F_u
\right)\\
&+\frac{\dee}{\dee z}
\left(
\tilde F_{\bar z}\tilde F_u-\tilde F_{\bar z}\tilde F_u^2
\right)\\
&+\frac{\dee}{\dee \bar z}
\left(
\tilde F_{z}\tilde F_u-\tilde F_{z}\tilde F_u^2
\right).
\end{align*}
From integration by parts we see that the last four terms integrate to zero. Consulting \eqref{E:iTerm} we find that
\begin{equation}\label{E:CubeSimp}
\intl_B \mu(F)\, dV
= \intl_B \left(1+(3\tilde F_{z\bar z}-1)\tilde F_u^2\right)\, dV.
\end{equation}

Thus if \eqref{E:AdHocHyp} holds we may combine 
\eqref{E:Holder1} and \eqref{E:CubeSimp} to obtain
\begin{equation*}
\meas(Z)
\le (2\Vol(B))^{2/3}(\Vol(B))^{1/3}
=\meas(Z_0)
\end{equation*}
as claimed.
\end{proof}

\begin{Remark}
If $\tilde F$ is not $\equiv 0$ then there must be points near the boundary of the support of $\tilde F$ where 
$(3\tilde F_{z\bar z}-1)\tilde F_u^2<1$.  It follows that 
equality holds in the conclusion of Theorem \ref{T:HeisSemiGlobal} only when $Z=Z_0.$
\end{Remark}

\begin{Prob}
Does Theorem \ref{T:HeisSemiGlobal} hold without the hypothesis \eqref{E:AdHocHyp}?  What happens for more general (non graph-like) perturbations?
\end{Prob}

\section{The isoperimetric problem} \label{S:iso}

\subsection{First variation} Returning to the assumptions of \S \ref{SS:IsoQuo} we seek hypersurfaces $Z$ that are stationary for the isoperimetric quotient $\quo(Z)$. 

\begin{Thm}[{[Ham, Theorem 19]}]\label{T:FirstVarQuoKappa}  
Let $Z\subset\C^2$ be the compact strongly pseudoconvex boundary of a bounded domain.   Then $Z$  is stationary for $\quo$  if and only if the invariant $\kappa$ described in \S \ref{SS:ChrisThesis} is constant on $Z$.
\end{Thm}

\begin{proof}
Suppose that $Z$  is stationary for $\quo$.  Then  the first variation of $\meas(Z)$ must vanish whenever the first variation of $\vol(Z)$ vanishes.  Working locally with $Z$ in graph form as in \S \ref{SS:FirstVarMax} we find that the first variation of $\vol(Z)$ is given by $-\intl_B \mathring F\,dV$ while the first variation of $\meas(Z)$ is given by $-\frac{2^{2/3}}{3}\intl_B 
\opa(F_0) \mathring F\,dV$. Hence 
\begin{equation}\label{E:pseudo-stat}
\intl_B 
\opa(F_0) \mathring F\,dV=0 \text{ whenever }\intl_B \mathring F\,dV=0.
\end{equation}
 It follows that $\opa(F_0)$ and hence $\kappa = \frac{3}{8} \,\opa(F_0)$ must be constant.   We note that the ratio of the first variation of $\meas(Z)$ to the first variation of $\vol(Z)$ is given by $\frac{2^{2/3}}{3} \opa(F_0) = \frac{2^{11/3}}{9}\,\kappa$.

For the converse, assume that $\kappa$ is constant.  Then reversing the above argument we see that the ratio of the first variation of $\meas(Z)$ to the first variation of $\vol(Z)$ is again given by $ \frac{2^{11/3}}{9}\,\kappa$.  (We check this first for perturbations with small support; it then follows for general perturbations by a partition of unity argument.)
Consider now perturbation of $Z$ by a family of constant-Jacobian (but not volume-preserving) holomorphic maps (dilations, for example).  
Then the first variation of $\log \quo(Z)$ must vanish; it follows that in this special case the first variation of $\log  \vol(Z)$ is equal to  $\frac{3}{2}$ times the first variation of $\log  \meas(Z)$  and hence that the  ratio of the first variation of $\meas(Z)$ to the first variation of $\vol(Z)$ is equal to $\frac{2}{3} \frac{\meas(Z)}{\vol(Z)}$.  Since this ratio is already fixed at $\frac{2^{11/3}}{9}\,\kappa$ we see that the ratio is equal to $\frac{2}{3} \frac{\meas(Z)}{\vol(Z)}$ for {\em any} perturbation. Reversing our reasoning we see that the first variation of $\log \quo(Z)$ must vanish in general; that is, $Z$  is stationary for $\quo$.
\end{proof}

\begin{Add1}  When $Z$  is stationary for $\quo$ then $\kappa=\frac{3}{2^{8/3}} \frac{\meas(Z)}{\vol(Z)}$ must be constant and {\em positive}.
\end{Add1}

We note that the condition \eqref{E:pseudo-stat} makes sense even when $Z$ is not the boundary of a bounded domain (provided that we restrict attention to compactly supported perturbations).  We can still view such hypersurfaces as being stationary for $\quo$ even through $\quo$ itself is not defined. With this more lenient interpretation of the problem we again conclude that  $Z$  is stationary for $\quo$ if and only if $\kappa$ is constant, though now it is possible for $\kappa$ to be negative.  

Examples of hypersurfaces with constant $\kappa$ include spheres
$|z|^2+|w|^2=R^2$ (with $\kappa=3\cdot 2^{-1/3}R^{-4/3}$)
and (non-compact) hypersurfaces of the form $|z|^2-|w|^2=R^2$ (with $\kappa=-3\cdot 2^{-1/3}R^{-4/3}$).

\begin{Thm}\label{T:Li}
Let $Z\subset\C^2$ be the compact strongly pseudoconvex boundary of a bounded domain. Then $Z$  is stationary for $\quo$ if and only if the domain bounded by $Z$ is a constant-Jacobian biholomorphic image of the unit ball.
\end{Thm}

\begin{proof}  The fact that constant-Jacobian holomorphic images of the unit sphere are stationary for $\quo$ follows from the invariance properties of our problem along with the fact that the sphere has  constant $\kappa$.  

For the converse we simply combine Theorem \ref{T:FirstVarQuoKappa} with  Corollary 1.2 from [Li] (refer also to \S \ref{SS:web-curv} below) which implies that any compact $Z$ with  constant positive $\kappa$ must bound a constant-Jacobian biholomorphic image of the unit ball.\end{proof}

\begin{Remark}
See [Ham] for {\em local} results characterizing constant-Jacobian holomorphic images of the unit ball.
\end{Remark}

\begin{Remark}
The result in Theorem \ref{T:Li} (or, more precisely, the result of Li quoted above) is reminiscent of Alexandrov's theorem [Ale] stating that a compact connected embedded hypersurface in $\R^n$ with constant mean curvature is a geometric sphere.  (Recall that a hypersurface is stationary for the euclidean isoperimetric problem if and only if it has constant mean curvature -- see for example II.1.3 in [Cha].) Wente [Wen] showed that this result fails for immersed hypersurfaces in $\R^3$ -- in fact, there exists an immersed torus in $\R^3$ with constant mean curvature.  This leads to the following question.

\begin{Q}
Let $Z$ be a compact immersed strongly pseudoconvex hypersurface in $\C^2$ that is stationary for $\quo$.  Does it follow that $Z$ must be the image of the unit sphere by a map extending to a constant-Jacobian holomorphic (but not necessarily injective) map of the unit ball?
\end{Q}

\end{Remark}

\subsection{Second variation for the sphere}
\label{S:2VarSphere}

Let
\begin{equation*}
\eta=\frac14\left(-\bar w\,dz\w dw\w d\bar z
+\bar z \, dz\w dw\w d\bar w
-w\,dz\w d\bar z\w d\bar w
+z\, dw\w d\bar z\w d\bar w\right).
\end{equation*}
Then $\eta\w d\rho =(S\rho)\omega_{\C^2}$, where $S$ is the radial vector field 
\begin{equation*}
S=z\frac{\dee}{\dee z} + w\frac{\dee}{\dee w} 
+\bar z\frac{\dee}{\dee \bar z} 
+\bar w\frac{\dee}{\dee\bar w}. 
\end{equation*}
Thus from \eqref{E:FeffMeasForm} we have
\begin{equation} \label{E:FeffStar1}
\sigma_Z
=
2^{4/3} \frac{M(\rho)^{1/3}}{S\rho}\eta.
\end{equation}

Suppose that $\rho$ takes the form
\begin{equation*}
\rho(z,w) = F(z,w)\left(|z|^2+|w|^2\right)-1
\end{equation*}
where $F$ is a radial function satisfying $SF\equiv 0$.
(Any star-shaped hypersurface admits a defining function of this form.) Then we may rewrite \eqref{E:FeffStar1} as
\begin{equation*} 
\sigma_Z
=
2^{1/3} \frac{M(\rho)^{1/3}}{F(z,w)\left(|z|^2+|w|^2\right)}\,\eta
=2^{1/3}M(\rho)^{1/3}\,\eta;
\end{equation*}
after checking that $M(\rho)$ is homogeneous of degree two we find that
\begin{align} \label{E:FeffStar2}
\meas(Z)
&=
2^{1/3} \intl_Z M(\rho)^{1/3}\,\eta\notag\\
&=
2^{1/3} \intl_{Z_0} F^{-7/3} M(\rho)^{1/3}\eta,
\end{align}
where $Z_0$ is the unit sphere.  (Note that $\eta$ restricts to euclidean surface area on $Z_0$.)

In place of the vector fields \eqref{E:VFHeis} we 
use
\begin{align} \label{E:BallVFDefs}
L &= \bar w \frac{\dee}{\dee z} -\bar z \frac{\dee}{\dee w}\notag\\
\bar L &= w \frac{\dee}{\dee \bar z} - z\frac{\dee}{\dee  \bar w}\\
T &=i\left( z \frac{\dee}{\dee z} + w \frac{\dee}{\dee  w}-
\bar z \frac{\dee}{\dee \bar z} - \bar w \frac{\dee}{\dee  \bar w}\right)\notag
\end{align}
which are divergence-free, tangent to the sphere and satisfy
\begin{align} \label{E:BallVFRelations}
L\bar L - \bar L L &= -i T\notag\\
LT-TL&=2iL\\
\bar LT-T\bar L &= -2i\bar L.\notag
\end{align}

Set $F=e^G$.

\begin{Lem}\label{L:MASph} 
On $Z_0$ we have
\begin{align}\label{E:MASph}
M(\rho)
= e^{3G}\Bigg( 1 &+ \frac{(L\bar L+ \bar LL) G}{2} \notag\\
&+\frac{1}{4} (TG)^2 + \Im(\bar L G \cdot LTG)\\
&+\frac{1}{4}
\biggl(
(TG)^2 \cdot \frac{(L\bar L+ \bar LL) G}{2} 
+ LG\cdot \bar L G\cdot T^2G\notag\\
&\qquad\qquad-2 \Re \left(\bar L G\cdot TG  \cdot LTG\right)
\biggr)
\Bigg).\notag
\end{align}
\end{Lem}

\begin{proof}
A lengthy but routine computation serves to verify that
\begin{align*}
M(\rho)
= e^{3G}\Bigg( 1 &+ \frac{(L\bar L+ \bar LL) G}{2} \notag\\
&+\frac{1}{4} (TG)^2 + \Im(\bar L G \cdot LTG)\\
&+\frac{1}{4}
\left(
(TG)^2 \cdot \frac{(L\bar L+ \bar LL) G}{2} 
+ LG\cdot \bar L G\cdot T^2G\right)
\\
&-\frac{1}{2}\Re \left(\bar LG\cdot (S-iT)G\cdot L(S+iT)G\right)\\
&+SG
\biggl(
\frac32 + \frac34 SG + \frac{(L\bar L+ \bar LL) G}{2}
+\frac18 SG \cdot (L\bar L+ \bar LL) G \\
&\qquad\qquad + \frac18 (SG)^2 + \frac18 (TG)^2
\biggr)\\
&+ \frac14 S^2G\cdot LG \cdot \bar LG
-\Re \left(\bar LG\cdot LSG\right)\Bigg).
\end{align*}
Recalling that $SG=0$ we obtain \eqref{E:MASph}.
\end{proof}

Since only tangential derivatives are cited in \eqref{E:MASph},  the formula only uses the values of  $G$ along $Z_0$.

Setting $G=\eps \mathring{G}$ and applying \eqref{E:FeffStar2} we obtain
\begin{align*}
&\meas(Z_\eps) 
=  
2^{1/3} \intl_{Z_0}
\Bigg( 
1 + \frac{\eps}{3}\left(\tfrac{(L\bar L+ \bar LL) \mathring{G}}{2} - 4\mathring{G}\right)\\
&+ \frac{\eps^2}{36}\Big(
3 |T\mathring{G}|^2 + 12 \Im \bar L\mathring{G}\cdot LT\mathring{G}
-2 |L\bar L \mathring{G}|^2
-2\Re (L\bar L\mathring{G})^2
-16 \mathring{G} \tfrac{(L\bar L+ \bar LL) \mathring{G}}{2}
+32 \mathring{G}^2
\Big)
\Bigg)
\,\eta\notag\\
&+O(\eps^3).\notag
\end{align*}
Using integration by parts along with \eqref{E:BallVFRelations} we find that
\begin{align}\label{E:Parts}
\intl_{Z_0}  \tfrac{(L\bar L+ \bar LL) \mathring{G}}{2} \,\eta 
&= 0\notag\\
2\Im \intl_{Z_0} \bar L\mathring G\cdot LT\mathring G\,\eta
&= i \intl_{Z_0} \left( L\bar L-\bar L L\right) \mathring G\cdot T\mathring G\,\eta
= \intl_{Z_0} |T\mathring G|^2 \,\eta\notag\\
2 \Re \intl_{Z_0} (L\bar L\mathring{G})^2\,\eta
&= 2 \intl_{Z_0} |L\bar L\mathring G|^2 \,\eta
-i \intl_{Z_0} \left( L\bar L-\bar L L\right)\mathring G\cdot T\mathring G \,\eta\\
&= 2 \intl_{Z_0} |L\bar L\mathring G|^2 \,\eta
- \intl_{Z_0} |T\mathring G|^2 \,\eta \notag\\
\intl_{Z_0} \mathring{G} \tfrac{(L\bar L+ \bar LL) \mathring{G}}{2} \,\eta 
&= -\intl_{Z_0}  \left| L\mathring G\right|^2 \,\eta\notag
\end{align}
and so
\begin{multline}\label{E:2VarFeffBall}
\meas(Z_\eps) 
=  \meas(Z_0) 
-
\frac{\eps \,2^{7/3}}{3} 
\intl_{Z_0}
\mathring{G} \,\eta\\
+
\frac{\eps^2}{9\cdot 2^{2/3}} 
\intl_{Z_0}
\left(
5|T\mathring{G}|^2 - 2 |L\bar L\mathring{G}|^2+8|L\mathring{G}|^2+16 \mathring{G}^2
\right)
\,\eta
+O(\eps^3).
\end{multline}

Using spherical coordinates we find that 
\begin{equation}\label{E:VolInt}
\vol(Z)=\frac{1}{4}\intl_{Z_0} F^{-2}\,\eta,
\end{equation}
 yielding
\begin{equation}\label{E:2VarVolBall}
\vol(Z_\eps)
=
\vol(Z_0)
- \frac{\eps}{2} \intl_{Z_0}
\mathring{G} \,\eta
+ \frac{\eps^2}{2} \intl_{Z_0}
\mathring{G}^2 \,\eta
+O(\eps^3).
\end{equation}

We can combine \eqref{E:2VarFeffBall} and \eqref{E:2VarVolBall} to obtain an expansion for $\quo(Z_\eps)$.  The result is simpler if we assume that $\intl_{Z_0}
\mathring{G} \,\eta=0$, in which case we obtain the following (with  help  from  \eqref{E:Parts}).
\begin{align}\label{E:2QVarVolBall}
\quo(Z_\eps)
&=
8\pi
+
\frac{\eps^2}{3\pi}
\intl_{Z_0}
\left(
5|T\mathring{G}|^2 - 2 |L\bar L\mathring{G}|^2+8|L\mathring{G}|^2-8 \mathring{G}^2
\right)
\,\eta
+O(\eps^3)\notag\\
&=
8\pi
+
\frac{\eps^2}{3\pi}
\intl_{Z_0}
\left(
5|T\mathring{G}|^2 - 2 |L\bar L\mathring{G}+2\mathring{G}|^2
\right)
\,\eta
+O(\eps^3)\\
&=
8\pi
+
\frac{\eps^2}{3\pi}
\intl_{Z_0}
\left(
\frac{9}{2}|T\mathring{G}|^2 
- 2 |\tfrac{(L\bar L+ \bar LL) \mathring{G}}{2} +2\mathring{G}|^2
\right)
\,\eta
+O(\eps^3).\notag
\end{align}

\begin{Lem}\label{L:Fourier} Let $U$ denote the  unit ball  in $\C^2$ centered at 0.  Then
\begin{align*}
\intl_{U} |z|^{2a}|w|^{2b}\,dV&=
\pi^2\frac{a!b!}{(a+b+2)!}\\
\intl_{Z_0} |z|^{2a}|w|^{2b}\,\eta&=
2\pi^2\frac{a!b!}{(a+b+1)!}.
\end{align*}
\end{Lem}

\begin{proof}
Using polar coordinates in each variable we obtain \[\intl_{B(0,R)} |z|^{2a}|w|^{2b}\,dV=
\pi^2\frac{a!b!}{(a+b+2)!}R^{2(a+b+2)}.\]  Differentiating with respect to $R$ we obtain the corresponding spherical integrals.
\end{proof}

\begin{Ex}\label{E:A} 
 Pick $(j,k)\ne(0,0)$ and let $\mathring G$ agree with $z^j w^k+\bar z^j \bar w^k$ along $Z_0$.    Using  \eqref{E:2QVarVolBall} and Lemma \ref{L:Fourier}  we have
\begin{equation*}
\quo(Z_\eps) 
=  8\pi
+
\eps^2\,\frac{16\pi}{3}
\frac{j! k!}{(j+k+1)!}
(j+k+2)(j+k-1)
+O(\eps^3).
\end{equation*}
So the second variation is positive when $j+k\ge2$ but 
vanishes for $(j,k)=(1,0)$ or $(0,1)$.
\end{Ex}

\begin{Ex}\label{E:B}  Pick $(j,k)$ with $j,k\ge1$ and let 
$\mathring G(z,w)$ agree with $z^j \bar w^k+\bar z^j  w^k$ along $Z_0$.   Computing as above we have
\begin{equation*}
\quo(Z_\eps) 
=  8\pi 
-\eps^2\,\frac{8\pi}{3}
\frac{j! k!}{(j+k+1)!}
(j+2)(j-1)(k+2)(k-1)
+O(\eps^3).
\end{equation*}
So the second variation is negative when both $j$ and $k$ are $\ge 2$ but vanishes when $j$ or $k$ is equal to $1$. 
\end{Ex}

Orthogonality considerations reveal that the second variation form is diagonal on the span of the functions examined in Examples \ref{E:A} and \ref{E:B}.
Thus we have established the following.

\begin{Thm} \label{T:Saddle}
The sphere $Z_0$ is a  saddle point for $\quo$. 
\end{Thm}

For $\mathring G(z,w)=z \bar w+\bar z  w$ we have $L\bar L\mathring{G}+2\mathring{G}=0=T\mathring{G}$.  Thus $\mathring{G}$ is a null function for the polarization of the second variation form, showing that the second variation form is degenerate. Note also that the computations above provide explicit infinite-dimensional subspaces along which the second variation form is positive definite or negative definite.

From Theorems \ref{T:Li} and \ref{T:Saddle} and the invariance properties of the problem we see that all critical points for $\quo$ are saddle points.

\subsection{Special families of domains}

\subsubsection{Circular hypersurfaces}  A {\em circular hypersurface} is a hypersurface intersecting each complex line through the origin in a circle centered at the origin.  A circular hypersuface $Z\subset \C^2$ admits a defining function  $e^{G(z,w)}\left(|z|^2+|w|^2\right)-1$ with $SG=TG=0$ where $S$ and $T$ are the vector fields from \S \ref{S:2VarSphere}.

From \eqref{E:2QVarVolBall} we see that the second variation for $\quo$ is negative semi-definite when restricted to circular perturbations of the unit sphere.  It turns out that the corresponding global result also holds.

\begin{Thm}\label{T:Circ}
For $Z$ a strongly pseudoconvex circular hypersurface we have
\begin{equation}\label{E:IsoperCirc}
\quo(Z)\le 8\pi 
\end{equation}
with equality holding if and only if $Z$ is the $\C$-linear image of a sphere.
\end{Thm}

\begin{proof}
From Lemma  \ref{L:MASph} we have
$M(\rho)
= e^{3G}\left( 1 + \tfrac{(L\bar L+ \bar LL) G}{2} \right)$.

Using \eqref{E:FeffStar2}, \eqref{E:VolInt}, 
\eqref{E:Parts} and H\" older's inequality we find that
\begin{align*}
\meas(Z)
&=2^{1/3}
\intl_{Z_0} 
e^{-4G/3} \sqrt[3]{1+\tfrac{(L\bar L+ \bar LL) G}{2}}\,\eta\\
&\le 2^{1/3}
\left(\,\intl_{Z_0} 
e^{-2G}\,\eta
\right)^{2/3}
\cdot
\left(\,\intl_{Z_0}
\left(1+\tfrac{(L\bar L+ \bar LL) G}{2}\right)\,\eta
\right)^{1/3}\\
&=2^{1/3} \left(4\vol(Z)\right)^{2/3} (2\pi^2)^{1/3}
=\left(8\pi \,\vol(Z)\right)^{2/3}
\end{align*}
which is equivalent to \eqref{E:IsoperCirc}. 

Equality will hold if and only if $e^{-2G}$ is a positive constant multiple of $1+\tfrac{(L\bar L+ \bar LL) G}{2}$.
To analyze this condition note that 
\begin{equation}\label{E:SphereMetric}
2e^{-G(z,w)}\,\frac{|z\,dw-w\,dz|}{|z|^2+|w|^2}
=
2e^{-G\left(1,\frac{w}{z}\right)}\,\frac{\left|d\left(\frac{w}{z}\right)\right|}{1+\left|\frac{w}{z}\right|^2}
\end{equation}
defines a sub-Riemannian metric on $\C^2\setminus\{0\}$ descending to a metric on the Riemann sphere $\CP^1$ with constant curvature
$
\kappa=e^{2G}\left(1+\tfrac{(L\bar L+ \bar LL) G}{2}\right).
$  (To verify this formula, note that the case $G=0$ corresponds  the standard spherical metric
and that the operator $\tfrac{L\bar L+ \bar LL}{2}$ acting on functions killed by $T$ is the lift to $S^3$ of the  the spherical Laplacian on $\CP^1$ via the Hopf fibration $S^3\to\CP^1$.)

It is well-known that any metric on $\CP^1$ with constant positive curvature is equivalent via a linear fractional transformation to a constant multiple of the standard spherical metric.  Computationally this implies that $e^{G(z,w)}$ must take the form
$\frac{|\alpha z + \beta w|^2+|\gamma z + \delta w|^2}{|z|^2+|w|^2}$.  Thus equality holds in \eqref{E:IsoperCirc} precisely when $Z$ is given by $|\alpha z + \beta w|^2+|\gamma z + \delta w|^2=1$. 

 (Note that the quantities  $\kappa$ and $\gamma$ used here are not the ones from \S \ref{SS:ChrisThesis}.) \end{proof}

\begin{Remark}\label{R:Curv}
For further geometric insight into the construction of the metric  \eqref{E:SphereMetric}, note that our circular hypersurface $Z$ may be viewed as the boundary of the unit tube  (with the zero section blown down) for some metric on the tautological bundle on $\CP^1$.  Since the tangent bundle on $\CP^1$ is the inverse square of the tautological bundle, the induced Riemannian metric on $\CP^1$ is (up to a constant) precisely the one given by \eqref{E:SphereMetric}.

Letting $dA$ denote the area form on $\CP^1$ induced by the metric \eqref{E:SphereMetric} and invoking the Gauss-Bonnet theorem $\intl_{\CP^1} \kappa\,dA=4\pi$, the proof of the inequality \eqref{E:IsoperCirc} may be rewritten as follows:
\begin{align*}
\meas(Z)
&= 2^{-2/3}\pi \intl_{\CP^1} \kappa^{1/3}\,dA\\
&\le 2^{-2/3}\pi
\left(\,  \intl_{\CP^1} dA \right)^{2/3}
\left( \, \intl_{\CP^1} \kappa\,dA \right)^{1/3}\\
&= 2^{-2/3}\pi
\left( \frac{8}{\pi} \vol(Z) \right)^{2/3}
\left(  4\pi \right)^{1/3}=\left(8\pi \,\vol(Z)\right)^{2/3}.
\end{align*}
(Here we have used the fact that integrals over $\CP^1$ with respect to the standard spherical metric pick up a factor of $\frac{\pi}{2}$ when lifted to $S^3$.)

\smallskip

Note also the connection between the quantity $\intl \kappa^{1/3}\,dA$ appearing above and \linebreak Blaschke's equiaffine surface area $\intl \kappa^{1/4}\,dA$.
\end{Remark}

\begin{Remark}
By constructing smooth strongly pseudoconvex approximations to the boundary of the bidisk we can construct circular domains with arbitrarily small values of $\quo(Z)$; for a somewhat related means to the same end, use the set-up of Remark \ref{R:Curv} and consider metrics on the sphere with positive curvature concentrated near a finite set.
\end{Remark}

\subsubsection{Intersections of balls}\label{SS:BallCap} We wish to study the isoperimetric quotient on the space of boundaries of non-trivial intersections of two balls.  

Up to affine equivalence, this space is parameterized by the angle of intersection $\theta\in(0,\pi)$ and the ratio of radii $R\in(0,1]$.  (Compare [BV, 
\S 2].)  A direct calculation shows that the isoperimetric quotient $\quo$ is given by 
\begin{equation*}
q(R,\theta)= 8 \sqrt{\pi}
\frac{
\left(R^{8/3} \lambda
+
\nu
-
\frac{R \sin\theta\left(1+R^{8/3}-
\left(R+R^{5/3}\right) \cos\theta\right) }{1+R^2-2 R \cos\theta}\right)^{3/2}
}
{
R^4 \lambda
+
\nu
 - 
 \frac{ R \sin\theta
 \left(
 1+\frac23 R^2\sin^2\theta+R^4-(R+R^3)\cos\theta
 \right) }
{
( 1 + R^2 - 2 R \cos\theta)
}},
\end{equation*}
where  
\begin{align*}
\lambda &= \arccos\left(\frac{R-\cos\theta }{\sqrt{1+R^2-2 R \cos\theta}}\right)\\
\nu &=  \arccos\left(\frac{1-R \cos\theta}{\sqrt{1+R^2-2 R \cos\theta}}\right).
\end{align*}
Of course here we have implicitly extended the definition of  $\meas$ to the case of hypersurfaces with  corners.  It is  easy to check  that the results of this extended definition agree with those obtained from taking a limit value of $\meas$ using a  standard exhaustion by smooth strongly pseudoconvex hypersurfaces.

The function $q(R,\theta)$ extends continuously (but not smoothly) to the region $0\le R\le 1, 0\le\theta<\pi$.  The boundary segments $\theta=0, \,\theta=\pi$ and $R=0$ correspond, respectively, to pairs of internally tangent balls (intersecting in a ball), to pairs of externally tangent balls (with empty intersection) and to intersections of balls with real half-spaces.

The expansions
\begin{equation*}
q(R,\theta)=8\pi-\frac{8\left(1-\sqrt[3]{R}\right)}{(1-R)^3}\theta^3+O\left(\theta^4\right)
\end{equation*}
for $0<R<1$ and
\begin{equation*}
q(R,\theta)=8\pi-\frac{8\theta^3}{3\left(\theta^2+(1-R)^2\right)}+O\left(\theta^2+(1-R)^2\right)
\end{equation*}
near  $R=1, \theta=0$ show that 
the spheres $\theta=0$ are locally maximizing in this setting, but they are not globally maximizing: indeed, $q(R,\theta)\to\infty$ as $\theta\nearrow \pi$.

 Numerical work indicates that the minimum value of $q(R,\theta)$ on $0\le R\le 1, 0\le\theta<\pi$ is $17.0297\dots$ occuring at $R=0,\,\theta=1.9473\dots$.

\subsubsection{Holomorphic images of balls}\label{S:HoloBall}

Let $H$  be a diffeomorphism of the closed unit ball $\bar U$  that is biholomorphic on the open ball $U$.  Let $Z_0=\bndry B, Z=H(Z_0)$.  Then we have
 \begin{subequations}\label{E:HoloBall}
\begin{align}
\meas(Z)&= 2^{1/3} \intl_{Z_0} |\det H'|^{4/3}\,\eta\\
\vol(Z)&= \intl_{U} |\det H'|^{2}\,dV. 
\end{align}
 \end{subequations}

First we show in Example \ref{X:shear} how to constuct examples in this class with large values of $\quo(Z)$, then we turn to consideration of lower bounds.

\begin{Ex}\label{X:shear} 
In the special case where $H$ has the form $(z,w)\mapsto (\phi(w)z,w)$   we have 
\begin{align*}
\meas(Z) 
&= 2^{4/3}\pi 
\intl_{|w|<1} \left|\phi(w)\right|^{4/3}
\,dA\\
\vol(Z) &= \pi 
\intl_{|w|<1} \left|\phi(w)\right|^{2} (1-|w|^2)
\,dA.\\
\end{align*}
(Formula (1.4.7c) from [Rud] is helpful in connection with the above computation of $\meas(Z)$.)

Setting $\phi_\eps(w)=(w-1-\eps)^{-3/2}$ we obtain the following by integrating over lines through $w=1$:
\begin{align*}
\meas(Z_\eps) &= 2^{4/3}\pi^2 |\log \eps| + O(1)\\
\vol(Z_\eps) &= 4\pi |\log \eps| + O(1)\\
\quo(Z_\eps) &= \pi^2 \sqrt{|\log \eps|} + O(1).
\end{align*}
 So $\quo$ can take on arbitrarily large values in this setting.  (Note also that the $Z_\eps$ just constructed are also invariant under rotation of the $z$-variable, showing that there is no analogue of Theorem \ref{T:Circ} for Hartogs domains; this also follows from the work in \S \ref{SS:BallCap}.)
\end{Ex}

\begin{Thm}\label{T:BIP}
For $Z=H(Z_0)$ as in the top paragraph of \S \ref{S:HoloBall} we have
\begin{equation}\label{E:BIP}
\quo(Z)\ge 4\pi.
\end{equation}
\end{Thm}

\begin{Q}\label{Q:BIP}
Can the constant $4\pi$ in Theorem \ref{T:BIP} be replaced by $8\pi$, with equality holding if and only if $H$ has the form $H_1\circ H_2$, where $H_2$ is an automorphism of the ball and $\det H_1'$ is constant?
\end{Q}

Let $h=\det H'$.  In view of \eqref{E:HoloBall}, Theorem \ref{T:BIP} will follow from the inequality
\begin{equation}\label{E:HL2}
\|h\|_{L^2(U)}\le \frac{1}{2^{3/4}\pi^{1/2}} \|h\|_{L^{4/3}(Z_0)},
\end{equation}
and a positive answer to Question \ref{Q:BIP} would follow from a sharpening of \eqref{E:HL2} to
\[\|h\|_{L^2(U)}\le \frac{1}{2^{5/4}\pi^{1/2}} \|h\|_{L^{4/3}(Z)}.\]

The inequality \eqref{E:HL2} without an explicit constant appears as Theorem 5.13 in [BB] (see also the proof of Theorem \ref{E:GenBIP} below).  Such inequalities may be viewed as generalizations of the Hardy-Littlewood inequality \[\|h\|_{L^{2p}(\{|z|<1\})}\le C_p \|h\|_{L^{p}(\{|z|=1\})}\] for holomorphic functions in one complex variable
[HL, Thm. 31] -- see  [Vuk] for sharp constants for the Hardy-Littlewood result  serving as the basis for a proof of the  planar isoperimetric inequality.

\begin{proof}[Proof of \eqref{E:HL2}] 
Let $X=z\frac{\dee}{\dee z} + w\frac{\dee}{\dee w} + 2$.  Note that $X=-iT+2=-\bar{L}L+2$ when applied to holomorphic functions.

Let $g$ be the holomorphic function solving $Xg=h$ on $U$. (The function $g$ can  be constructed by an easy power series computation, or see [Bar2, Lemma 3] for a somewhat more general argument.)

Using integration by parts and H\"older's inequality we obtain
\begin{align}\label{E:gh}
\intl_U |h|^2\,dV
&=  \intl_U X(g\bar h)\,dV\notag\\
&=  \frac{1}{2} \intl_{Z_0} g\bar h\,\eta \\
&\le \frac{1}{2} \|g\|_{L^4(Z_0)}  \|h\|_{L^{4/3}(Z_0)}.\notag
\end{align}

To derive an estimate for $\|g\|_{L^4(Z_0)}$ we quote the sub-Riemannian Sobolev inequality of Jerison and Lee 
[JL2] to obtain 
\begin{equation}\label{E:JL} 
\|g\|_{L^4(Z_0)}^{2}
\le
\frac{1}{\sqrt{2}\pi} \intl_{Z_0} \left( |g|^2+|Lg|^2 \right)\,\eta.
\end{equation}
(This inequality may be obtained by setting $u=|g|$ in the ball version of the Jerison-Lee inequality as formulated on p.\,174 of [JL1], or see \S \ref{SS:JL-Sob} below.)
But
\begin{align*}\label{E:JL-rhs}
 \intl_{Z_0} \left( |g|^2+|Lg|^2 \right)\,\eta
&= \intl_{Z_0} \left((1-\bar{L}L)g\right)\bar g\,\eta\notag\\
&= \intl_{Z_0} (h-g) \bar g\,\eta\\
&\le \|h\|_{L^{4/3}(Z_0)} \|g\|_{L^{4}(Z_0)} - \|g\|_{L^{2}(Z_0)}^2.\notag
\end{align*}
Combining this with \eqref{E:JL} we obtain
\begin{equation*}
\|g\|_{L^4(Z_0)}^{2}
\le \frac{1}{\sqrt{2}\pi}  \|h\|_{L^{4/3}(Z_0)} \|g\|_{L^{4}(Z_0)}
\end{equation*}
and so
\begin{equation}\label{E:g4}
\|g\|_{L^4(Z_0)}
\le \frac{1}{\sqrt{2}\pi}  \|h\|_{L^{4/3}(Z_0)}.
\end{equation}

Combining \eqref{E:g4} with \eqref{E:gh} we have
\begin{equation*}
\intl_U |h|^2\,dV
\le 
\frac{1}{2^{3/2} \pi} \|h\|_{L^{4/3}(Z_0)}^2
\end{equation*}
yielding \eqref{E:HL2}.
\end{proof}

\section{A biholomorphically-invariant isoperimetric constant}\label{S:inv-iso}

For $Z=\bndry\Omega$ satisfying the assumptions of \S \ref{SS:IsoQuo} the isoperimetric quotient $\quo(Z)$ is not invariant under biholomorphic mapping with non-constant Jacobian, but we may form a genuine biholomorphic invariant as follows.

\begin{Def}
Let 
\begin{equation*}
\quo^*(Z)
=
\inf \{\quo(H(Z))\st H \text{ diffeomorphic on $\bar\Omega$ and holomorphic in $\Omega$ }\}.
\end{equation*}
\end{Def}

\begin{Thm}\label{E:GenBIP}
For $Z$ as above the invariant $\quo^*(Z)$ is strictly positive.
\end{Thm}

\begin{proof}
As in  the proof of Theorem \ref{T:BIP} 
it suffices to have an inequality of the form 
\begin{equation*}\label{E:HL2-gen}
\|h\|_{L^2(\Omega)}\le C_{\Omega} \|h\|_{L^{4/3}(Z)},
\end{equation*}
where $\Omega$ is the domain enclosed by $Z$.    This inclusion estimate is proved in [Bea,Thm. 1.5(iii)]  (see also [CK, Thm. 1.1]).
\end{proof}

\begin{Q} \label{Q:MinMax}
Does the sphere maximize $\quo^*(Z)$?
\end{Q}

Affirmative answers to Questions \ref{Q:BIP} and \ref{Q:MinMax} would cast additional light on the saddle point behavior from Theorem \ref{T:Saddle}.

\begin{Q}\label{Q:Ext} 
Are extremals guaranteed to exist in the definition of $\quo^*(Z)$?
\end{Q}

\begin{Remark}
Theorem \ref{E:GenBIP} also holds in higher dimension, since the required estimate $\|h\|_{L^2(\Omega)}\le C_{\Omega} \|h\|_{L^{2n/(n+1)}(Z)}$ is also covered by 
[Bea,Thm. 1.5(iii)].  Questions \ref{Q:MinMax} and \ref{Q:Ext} are also of interest in this setting.
\end{Remark}

\section{Normalization issues}\label{S:normalize}
\subsection{A note on the definition of Fefferman's measure}\label{SS:Fef-def-note}

The original definition of Fefferman's measure in [Fef1] includes an unspecified dimensional constant.  In [Bar1] the first author proposed the choice made in \eqref{E:FeffMeasForm} above in order to maximize compatibility with Blaschke's constructions in real affine geometry (as discussed in \S \ref{S:spec} above).  Other choices have been used elsewhere -- for instance, in [HKN]
the  factor of $2^{ 2n/(n+1)}$ is omitted.

\subsection{Webster curvature}\label{SS:web-curv}
The Webster theory is based on the choice of a contact form $\theta$ for $Z$.  To simplify the discussion we assume that the defining function $\rho$ satisfies Fefferman's approximate Monge-Amp\`ere equation 
\[- \det
\begin{pmatrix}
\rho & \rho_{z_j}\\
\rho_{z_{\bar k}} & \rho_{z_{j}\bar z_{  k}}
\end{pmatrix}=1+O(|\rho|^3).\] (Such a $\rho$ always exists [Fef2].)

In [Ham]  $\theta$ is chosen to be $2^{-4/3}i(\dee\rho-\deebar\rho)$ and the Webster curvature is found to be $-\frac{2^{5/3}}{3}\kappa$ in general and $-\frac{2^{4/3}}{R^{4/3}}$ on the sphere of radius $R$.

In [LiLu1] and [Li], on the other hand, $\theta$ is chosen to be $-\frac{i}{2}(\dee\rho-\deebar\rho)$ leading to curvature values which are $-2^{-1/3}$ times those in [Ham]; thus the Webster curvature is now $\frac{2^{4/3}}{3}\kappa$ in general and $\frac{2}{R^{4/3}}$ on the sphere of radius $R$.  Also, with this choice of $\theta$, the main formula in [LiLu2] can be used to check that the absolute value of the torsion coefficient is $2^{2/3}\gamma$.

\subsection{Jerison-Lee Sobolev inequality}\label{SS:JL-Sob}  It is instructive to set up the holomorphic Jerison-Lee Sobolev inequality \eqref{E:JL}  from the point of view of Li and Luk in [LiLu1].  
Setting $M=Z_0$, $\rho=|g(z,w)|^2(1-|z|^2-|w|^2)$, $\theta=-\frac{i}{2}(\dee\rho-\deebar\rho)$  and assuming at first that $g$ is zero-free we find with the use of Theorem 1.1 in [LiLu1] that the quotient on the right-hand side of (1.3) in [LiLu1] may be written in the form
\begin{multline*}
\frac{2\intl_{Z_0}
\left( |g|^2+
ig(T\bar{h})-i\bar{g}(Tg)-|Lg|^2
\right) \eta}
{\left(
2\intl_{Z_0} |g|^4\,\eta 
\right)^{1/2}}
=\sqrt{2}\;
\frac{\intl_{Z_0}
\left( |g|^2 -
g(L\bar{L}\bar{g})-\bar{g}(\bar{L}Lg)-|Lg|^2
\right) \eta}
{\left(
\intl_{Z_0} |g|^4\,\eta 
\right)^{1/2}},
\end{multline*}
where $\eta,T,L$ and $\bar{L}$ are as in \S \ref{S:2VarSphere}. 
After integrating the middle terms of the numerator by parts, the quotient reduces to
\begin{equation}\label{E:JL-quotient}
\sqrt{2}\;
\frac{\intl_{Z_0}
\left( |g|^2 +|Lg|^2
\right) \eta}
{\left(
\intl_{Z_0} |g|^4\,\eta 
\right)^{1/2}}.
\end{equation}

From Corollary B in [JL2] we have that \eqref{E:JL-quotient} is minimized when $g$ is constant; that is, \begin{equation*}
\sqrt{2}\pi \left(
\intl_{Z_0} |g|^4\,\eta 
\right)^{1/2}
\le 
\intl_{Z_0}
\left( |g|^2 +|Lg|^2
\right) \eta
\end{equation*}
which is equivalent to \eqref{E:JL}.

When $g$ has zeros the same conclusion may be obtained by setting \[\rho_\eps=\left(|g(z,w)|^2+\eps\right)\left(1-|z|^2-|w|^2\right)\] and letting $\eps$ decrease to 0.


\end{document}